\documentclass[12pt,letterpaper]{amsart}
\usepackage{geometry}
\usepackage{hyperref}
\usepackage{mathpazo} 
\usepackage{amsmath,amssymb,amsfonts,amsthm,amscd}
\usepackage{mathrsfs,latexsym,url,setspace,xcolor}
\newtheorem{theorem}{Theorem}
\newtheorem*{theorem*}{Theorem}

\newtheorem{lemma}{Lemma}

\theoremstyle{remark}

\newtheorem{remark}{Remark}
\newcommand{\C}{\mathbb{C}}
\newcommand{\D}{\Omega}
\newcommand{\ep}{\varepsilon}
\newcommand{\Dc}{\overline{\Omega}}
\newcommand{\dbar}{\overline{\partial}}
\newcommand{\zb}{\overline{z}}
\newcommand{\xb}{\overline{\xi}}
\newcommand{\wb}{\overline{w}}

\title{Essential norm estimates for Hankel operators on convex 
domains in $\C^2$}
\author{\v{Z}eljko \v{C}u\v{c}kovi\'c}
\author{S\"{o}nmez \c{S}ahuto\u{g}lu}
\email{Zeljko.Cuckovic@utoledo.edu, Sonmez.Sahutoglu@utoledo.edu}
\address{University of Toledo, Department of Mathematics \& Statistics, 
Toledo, OH 43606, USA}
\date{\today}
\begin{document}
\begin{abstract}
Let $\D\subset \C^2$ be a bounded convex domain with $C^1$-smooth boundary 
and $\varphi\in C^1(\Dc)$ such that $\varphi$ is harmonic on the nontrivial 
disks in the boundary. We estimate the essential norm of the Hankel operator 
$H_{\varphi}$ in terms of the $\dbar$ derivatives of $\varphi$ ``along'' the 
nontrivial disks in the boundary.  
\end{abstract}

\onehalfspace
\maketitle
 
Let $\D$ be a domain in $ \C^n$ for $n\geq 1$ and $b\D$ denote the boundary of 
$\D$. Furthermore, let $dV$ denote the volume measure on $\D$ and $A^2(\D)$ be 
the Bergman space on $\D$, the space of square integrable holomorphic functions 
on $\D$ with respect to $dV$.  The Bergman projection $P$ is the orthogonal 
projection from $L^2(\D)$ onto $A^2(\D)$. For $\varphi\in L^{\infty}(\D)$ we 
define the Hankel operator $H_{\varphi}:A^2(\D)\to L^2(\D)$ by 
\[H_{\varphi}f=(I-P)(\varphi f)\] 
where $I$ denotes the identity operator on $L^2(\D)$. 

In \cite{CuckovicSahutoglu09} we studied compactness of Hankel operators 
on smooth bounded pseudoconvex domains with the symbols smooth up to the 
boundary.  Our most complete result is attained on smooth bounded convex domains
in $\C^2$. On such domains we characterize compactness of $H_{\varphi}$ in 
terms of the behavior of $\varphi$ on the analytic disks in $b\D$.  Throughout 
this paper $\mathbb{D}$ will denote the unit open disk in $\C$. 

\begin{theorem*}[\cite{CuckovicSahutoglu09}]
 Let $\D$ be a smooth bounded convex domain in $\C^2$ and 
 $\varphi\in C^{\infty}(\Dc)$. 
Then $H_{\varphi}$ is compact  if and only if $\varphi\circ F$ is holomorphic 
for all holomorphic $F:\mathbb{D}\to b\D$.
\end{theorem*}

In this paper we continue our study of compactness of Hankel operators and 
obtain estimates on their essential norms.  The essential norm $\|T\|_e$ of 
a bounded linear operator  $T:X\to Y$ where $X$ and $Y$ are normed linear 
spaces, is defined as 
\[\|T\|_e=\inf \Big\{\|T-K\|:K:X\to Y \text{ is a  compact linear 
operator}\Big\}.\]
That is, the essential norm of $T$  is the distance from $T$ to the  subspace 
of compact operators. 

The first estimates for the essential norms of Hankel operators were 
obtained by Lin and Rochberg \cite{LinRochberg93} in 1993, for the case of the 
Bergman space on $\mathbb{D}$. They showed that the essential norm 
estimates of $H_{\varphi}$, acting on $A^2(\mathbb{D})$, are analogous to the 
estimates on the Hardy space which is a famous 
theorem of Adamjan, Arov and Kre\u{\i}n \cite{AdamjanArovKrein71}.  
The Lin-Rochberg results were later generalized by Asserda \cite{Asserda00} to 
higher dimensions the case the domain is a strongly pseudoconvex. 

As in \cite{CuckovicSahutoglu09} our approach uses the connection between 
Hankel operators and the $\dbar$-Neumann operator. Due to this connection, we 
are able to consider more general domains; however, our symbols are more 
restricted. As a result, our estimates are of different type than Lin and 
Rochberg's estimates. In our case, the estimates depend on the behavior of the 
symbol on the analytic disks in the boundary of domains. We note that an 
analytic disk in the boundary of $\D$ is the image of a holomorphic function 
$F:\mathbb{D}\to b\D$. 

Before we state our main result we define $\Gamma_{b\D}$, the set of all linear 
parametrizations of ``circular'' affine nontrivial analytic disks in $b\D$, as 
follows:   
\[\Gamma_{b\D}=\Big\{F:\mathbb{D}\to b\D: F(\xi)=\xi z+p \text{ for some } 
p\in b\D, z\in \C^n\setminus \{0\}\Big\}.\]
We note that in case there are no nontrivial affine disks in the boundary of 
$\D$, the set  $\Gamma_{b\D}$ is empty.

In the main result below and the rest of the paper, $f_z$ and $f_{\zb}$ denote  
the derivative of $f$ with respect to $z$ and $\zb$ respectively. 

\begin{theorem}\label{ThmMain}
 Let $\D$ be a $C^1$-smooth bounded convex domain in $\C^2,\tau_{\D}$ 
denote the diameter of $\D$, and $\varphi\in C^1(\Dc)$ such that 
$\varphi \circ F$ is harmonic for every 
holomorphic $F:\mathbb{D}\to b\D$. Then the Hankel operator 
$H_{\varphi}$ satisfies the following essential norm estimate: 
\[\sup_{F\in \Gamma_{b\D}}\left\{\frac{|F'(0)|}{\sqrt{2}\tau_{\D}}
\inf_{\xi\in \mathbb{D}}\left\{ |(\varphi \circ F)_{\xb}(\xi)| \right\} \right\}
\leq \left\|H_{\varphi}\right\|_e
\leq  \sup_{F\in \Gamma_{b\D}}\left\{\frac{\sqrt{e}\tau_{\D}}{|F'(0)|} 
\sup_{\xi\in \mathbb{D}}\left\{ |(\varphi \circ F)_{\xb}(\xi)| 
\right\} \right\}.\]
\end{theorem}

\begin{remark}
Both estimates  in the theorem above are defined to be zero 
in case $\Gamma_{b\D}=\emptyset$. That is, in case there are no  nontrivial 
analytic disks in $b\D$ we get $\|H_{\varphi}\|_e=0$. This is in 
accordance with the fact that, in this case, $H_{\varphi}$ is compact. 
\end{remark}

\begin{remark}
$F'(0)$ measures the size of the disk $F(\mathbb{D})\subset b\D$. So 
it is interesting that the essential norm depends on the ``bar'' derivatives of 
$\varphi$ on the disks in the boundary as well as the size of these disks.
\end{remark}

In case of the bidisk we get a better estimate for the lower bound as in 
the following theorem. 

\begin{theorem}\label{ThmBidisk}
 Let $\varphi\in C^1(\overline{\mathbb{D}^2})$ such that the functions $z\to 
\varphi(z,e^{i\theta})$ and $w\to \varphi(e^{i\theta},w)$ are harmonic on 
$\mathbb{D}$ for all $\theta\in [0,2\pi]$. Then the Hankel operator 
$H_{\varphi}$ satisfies the following essential norm estimate: 
\[ \left\|H_{\varphi}\right\|_e  \geq \sup_{F\in 
\Gamma_{b\mathbb{D}^2}}\left\{\frac{|F'(0)|}{\sqrt{2}}
\inf_{\xi\in \mathbb{D}}\left\{ |(\varphi \circ F)_{\xb}(\xi)| \right\} \right\}
\]
\end{theorem}

\begin{remark}
The diameter  of the bidisk $\tau_{\mathbb{D}^2}=2\sqrt{2}$ is the 
distance between $(-1,-1)$ and $(1,1)$. Hence  
$\sqrt{2}\tau_{\mathbb{D}^2}=4>\sqrt{2}$. Thus the lower 
bound in Theorem \ref{ThmBidisk} is better than the one in 
Theorem \ref{ThmMain}. 
\end{remark}

\section*{Proofs of Theorem \ref{ThmMain} and \ref{ThmBidisk}}

\begin{lemma}\label{Lem1}
 Let $\gamma\in C^{1}_0(U)$ where $U\subset \mathbb{D}$ is a domain. 
 Then $\|\gamma_{\xi}\|=\|\gamma_{\xb}\|$.
\end{lemma}

\begin{proof} Since $\gamma$ is compactly supported in $U$  there are no
boundary terms in the following integration by parts formula.
\begin{align*} 
\|\gamma_{\xi}\|^2=\int_U\gamma_{\xi}(\xi)\overline{\gamma}_{\xb}(\xi) dV(\xi)
=\int_U\gamma(\xi)\overline{\gamma}_{\xb \xi}(\xi)dV(\xi) 
=\int_U\gamma_{\xb}(\xi)\overline{\gamma}_{\xi} (\xi)dV(\xi) 
=\|\gamma_ { \xb } \|^2.
\end{align*}
Therefore, $\|\gamma_{\xi}\|=\|\gamma_{\xb}\|$.
\end{proof}
 
 We note that a unitary affine mapping $F$ on $\C^n$ is of the form $F(z)=Az+p$ 
where $A$ is a $n\times n$ unitary matrix and $p\in \C^n$.
 \begin{lemma}\label{Lem2}
Let $V$ be a bounded domain in $\C^n,F$ be a unitary affine mapping, and 
$\phi\in L^{\infty}(V)$. Then 
$\|H_{\phi}\|_e=\|H_{\phi\circ F}\|_e$ where $H_{\phi\circ F}$ is the Hankel 
operator (with symbol $\phi\circ F$) on $A^2(F^{-1}(V))$.
 \end{lemma}
\begin{proof} 
Let $U=F^{-1}(V)$ and the pull-back $F^*:A^2(V)\to A^2(U)$ be defined as 
$F^*(f)=f\circ F$ for $f\in A^2(V)$. Then one can check that $F^*$ is an 
isometry. Furthermore, the Bergman kernel transformation formula of 
Bell (see, \cite[Proposition 6.1.7]{JarnickiPflugBook}) 
$P^V=(F^{-1})^*P^UF^*$ where $P^U,P^V$ are the Bergman projections on $U$ and 
$V$, respectively.  Then for $f\in A^2(V)$ we have 
\begin{align*}
 (F^{-1})^*H_{\phi\circ F}F^*(f)&= (F^{-1})^*H_{\phi\circ F}(f\circ F)\\
&=  (F^{-1})^*\Big(\phi(F)f(F)- P^U(\phi(F)f(F))\Big)\\
&=\phi f- (F^{-1})^*P^UF^*(\phi f)\\
&=\phi f- P^V(\phi f)\\
&=H_{\phi}(f)
\end{align*}
 Also $T^V:A^2(V)\to L^2(V)$ is a compact linear operator if and only if 
$T^U:A^2(U)\to L^2(U)$ is compact where $T^V= (F^{-1})^*T^UF^*$. 
 Furthermore, 
 \[ \left\|H_{\phi}-T^V\right\|
 =\left\|(F^{-1})^*H_{\phi\circ F}F^*-(F^{-1})^*T^UF^*\right\|
 =\left\|H_{\phi\circ F}-T^U\right\|.\]
One can check that, the equality above implies that 
$\left\|H_{\phi}\right\|_e=\left\|H_{\phi\circ F}\right\|_e$.
\end{proof}

We will use the $\dbar$-Neumann problem to obtain the upper bound in 
Theorem \ref{ThmMain}. The $\dbar$-Neumann operator, denoted by $N$, is defined 
as the solution operator for the complex Laplacian 
$\dbar\dbar^*+\dbar^*\dbar$ on square integrable 
$(0,1)$-forms on $\D$,  denoted by $L^2_{(0,1)}(\D)$. We refer the reader to 
the books \cite{ChenShawBook,StraubeBook} and references there in, for more 
information about the $\dbar$-Neumann problem. In the following theorem we 
list the properties we need  about $N$ 
(see \cite[Theorem 4.4.1]{ChenShawBook}). 
\begin{theorem*}
 Let $\D$ be a bounded pseudoconvex domain in $\C^n$ for $n\geq 2$. There 
exists a bounded self-adjoint operator $N:L^2_{(0,1)}(\D) \to L^2_{(0,1)}(\D)$ 
such that 
\begin{itemize}
 \item[i.] $(\dbar^*\dbar+\dbar\dbar^*)N=I$ on $L^2_{(0,1)}(\D)$,
 \item[ii.] $\dbar^*N$ is the solution operator to 
$\dbar u=v$ that produces solutions orthogonal to $A^2(\D)$,  
 \item[iii.] the Bergman projection $P$ satisfies the following equality 
 \[P=I-\dbar^*N\dbar\]
 where $I$ is the identity mapping,
\item[iv.] the operators $\dbar N,\dbar^* N, \dbar\dbar^*N,$ and $\dbar^*\dbar 
N$ are bounded and 
\[\|N\| \leq e\tau_{\D}^2,\quad  \|\dbar N\|\leq \sqrt{e}\tau_{\D},\quad  
\|\dbar^* N\|\leq \sqrt{e}\tau_{\D}\] 
where $\tau_{\D}$ is the diameter of $\D$.
 \end{itemize}
\end{theorem*}

We note that ii. and iii. in the theorem above imply that $\dbar 
H_{\varphi}f=f\dbar\varphi$ for $\varphi\in C^1(\Dc)$ and $f\in A^2(\D)$.

\begin{remark}
Before we start the proof of Theorem \ref{ThmMain} we note that 
even though \cite[Corollary 2]{CuckovicSahutoglu09} is stated for 
$\varphi\in C^{\infty}(\Dc)$, observation of the proof reveals that it is 
enough to have $C^1$-smooth domain $\D$  and $\varphi\in C^1(\Dc)$.
\end{remark}

\begin{proof}[Proof of Theorem \ref{ThmMain}]
First we  will prove the lower bound. Since $\D$ is a $C^1$-smooth bounded 
convex domain in $\C^2$, \cite[Corollary 2]{CuckovicSahutoglu09} implies that  
$H_{\varphi}$ is compact if and only if $\varphi \circ F$ is holomorphic for 
any holomorphic $F:\mathbb{D}\to b\D$. Thus, in order to find the essential 
norm estimate, without loss of generality, we assume that there exists 
holomorphic $F:\mathbb{D}\to b\D$ such that 
$(\varphi\circ F)_{\xb} \not\equiv 0$. Since $\varphi$ is $C^1$-smooth, this 
means that $\varphi_{\zb}(F)\neq 0$ on some open set. But the domain $\D$ is 
convex which implies that the disk $F(\mathbb{D})$ is an affine disk (see 
\cite{FuStraube98,CuckovicSahutoglu09}). Using Lemma \ref{Lem2} we can thus 
assume that there exists $\tau_0\in (0,\tau_{\D})$ such that
\begin{itemize}
\item[i.]  $\varphi_{\zb}(z,0)\neq 0$ for all $|z|\leq \tau_0$, 
 \item[ii.] $\{(z,w)\in \C^2:|z|\leq \tau_0, w=0\} \subset b\D$.
\end{itemize}
Since $\D$ is bounded we can also deduce that 
\begin{itemize}
\item[iii.] $\D\subset \{z\in \C:|z|<\tau_{\D}\}\times \{w\in \C:|w|< 
\tau_{\D}, \text{Re}(w)>0\}$. 
\end{itemize}
With this setup, we can now put a wedge $W$ in $\D$ perpendicular to  $D=\{z\in 
\C:|z|<\tau_0\}$. Furthermore, $W$ can be chosen as close to flat as we want if 
we are willing to choose its radius very small. That is, for any $\ep_1>0$ there 
exists $r_0>0$ so that $D\times W\subset \D$ where 
\[W=\left\{re^{i\theta}\in \C:0\leq 
r<r_0,|\theta|<\frac{\pi-\ep_1}{2}\right\}.\]
 Let us choose 
\[\chi(z)=\frac{2}{\pi \tau_0^2}\left(1-\frac{|z|^2}{\tau_0^2}\right) \text{ 
for } z\in \overline{D}.\]
Then $\chi\in C^{\infty}(\overline{D}), \chi\geq 0, \chi(z)=0$ for 
$|z|=\tau_0$. Then we have 
\[\int_D\chi(z) dV(z)
=\frac{2}{\pi \tau_0^2} 2\pi \int_0^{\tau_0}
\left(\rho-\frac{\rho^3}{\tau_0^2}\right)d\rho
=\frac{4}{\tau_0^2}
\left(\frac{\tau_0^2}{2}-\frac{\tau_0^4}{4\tau_0^2}\right)=1,\]
and 
\[\|\chi_z\|^2=\frac{4}{\pi^2 \tau_0^4}
\int_D\frac{|z|^2}{\tau_0^4}dV(z)=\frac{4}{\pi^2\tau_0^8}2\pi
\int_0^{\tau_0}\rho^3d\rho=\frac{8}{\pi 
\tau_0^8}{\frac{\tau_0^4}{4}}=\frac{2}{\pi \tau_0^4}.\]
Hence
\[\frac{\int_D\chi(z)dV(z)}{\|\chi_z\|}=\tau_0^2\sqrt{\frac{\pi}{2}}
=\frac{V(D)}{\sqrt{2\pi}}.\]
Let us first restrict $\varphi$ onto $D$ and extend the 
restriction as a  $C^1$-smooth function $\phi_1$ defined on  
$\C\times\{0\}$. Finally, we extend the function  $\phi_1$ trivially as a 
$C^1$-smooth function $\varphi_1$ on $\C^2$. That is, 
$\varphi_1(z,w)=\varphi(z,0)$.  
Let us define $\varphi_2=\varphi-\varphi_1$ and 
\[\gamma(z)=\frac{\chi(z)}{\varphi_{1\zb}(z,0)} \text{ for } z\in 
\overline{D}.\] 
where $\varphi_{1\zb}$ denotes $\frac{\partial \varphi_{1}}{\partial 
\zb}$. We will continue to use this notation below when appropriate. 

We note that, in the rest of the proof $\|.\|$ and $\|.\|_U$ denote the $L^2$ 
norm on $\D$ and on open set $U$, respectively. 

Let us define $\alpha_j=1-2^{-2j-1}$ and 
\[f_j(z,w)=\frac{1}{2^jw^{\alpha_j}} \text{ for } (z,w)\in \D.\] 
Using polar coordinates one can show that 
\begin{align}\label{Eqn0}
\|f_j\|_{W}= \sqrt{\pi-\ep_1}r_0^{1-\alpha_j} \text{ and }   
\|f_j\| \leq  \pi \tau_{\D}^{2-\alpha_j}. 
\end{align}
We will use the following equality in the second equality in 
\eqref{Eqn5} below.  
\[\frac{\partial H_{\varphi_1}f_j}{\partial \zb}d\zb+\frac{\partial 
H_{\varphi_1}f_j}{\partial \wb}d\wb=\dbar 
H_{\varphi_1}f_j=\dbar(\varphi_1f_j-P(\varphi_1f_j))=f_j\dbar\varphi_1
=f_j\frac{\partial \varphi_1}{\partial \zb}d\zb.\]
Then,  for $w\in W$ we have 
\begin{align}
 \nonumber \frac{1}{2^jw^{\alpha_j}} \int_{D}\chi(z)dV(z) 
 =& \int_{D} f_j(z,w)\frac{\partial \varphi_{1}}{\partial 
\zb}(z,w)\gamma(z)dV(z)\\
\label{Eqn5}=& \int_{D} \frac{\partial H_{\varphi_1}f_j}{\partial 
\zb}(z,w)\gamma(z) dV(z)\\
\nonumber =& - \int_{D}H_{\varphi_1}f_j(z,w)\gamma_{\zb}(z)dV(z).
\end{align}
We note that in the last equality above we used integration by parts and the 
fact that $\gamma(z)=0$ for $|z|=\tau_0$. 

Now we take the absolute values of both sides of \eqref{Eqn5}  and then apply 
the Cauchy-Schwarz inequality to the right hand side to get 
\[ |f_j(0,w)|\int_{D} \chi(z) dV(z) \leq \left\|H_{\varphi_1}f_j\right\|_{D} 
\|\gamma_{\zb}\|_{D}.\]
After integrating over the wedge $W$ and dividing by $\|\gamma_{\zb}\|_{D}$
we get
\[\|f_j(0,.)\|_W \frac{\int_{D}\chi(z) dV(z)}{\|\gamma_{\zb}\|_{D}} \leq 
\left\|H_{\varphi_1}f_j\right\|_{D\times W}
\leq \left\|H_{\varphi_1}f_j\right\|.\]
We remind the reader that $\varphi_{1\zb}$ and $\varphi_{1z\zb}$ below 
will denote $\frac{\partial \varphi_{1}}{\partial \zb}$ and 
$\frac{\partial^2\varphi_{1}}{\partial z\partial \zb}$ respectively.
Since we assumed that $\varphi$ is harmonic on $D$, Lemma \ref{Lem1} implies 
that 
\begin{align*}
 \|\gamma_{\zb}\|_{D}=  \|\gamma_{z}\|_{D}= \left\|
\frac{\chi_z}{\varphi_{1\zb}}-\chi
\frac{\varphi_{1z\zb}}{(\varphi_{1\zb})^2}\right\|_{D}=\left\|
\frac{\chi_z}{\varphi_{1\zb}}\right\|_{D}\leq
\frac{\|\chi_z\|_{D}}{\inf_{D}|\varphi_{1\zb}|}.
\end{align*}
Then
\begin{align}\label{Eqn1}
 \left\|H_{\varphi_1}f_j\right\| \geq 
\frac{\int_{D}\chi(z) dV(z)}{\|\gamma_{\zb}\|_{D}}\|f_j(0,.)\|_W\geq
\frac{\int_{D}\chi(z)dV(z)}{\|\chi_z\|_{D}}
\left(\inf_{D}|\varphi_{1\zb}|\right) 
\|f_j(0,.)\|_W.
\end{align}
Therefore, inequality \eqref{Eqn1} and the fact that 
$\|f_j(0,.)\|_W=\sqrt{\pi-\ep_1}r_0^{1-\alpha_j}$ imply that 
\begin{align} \label{Eqn2}
\left\|H_{\varphi_1}f_j\right\| \geq r_0^{1-\alpha_j} 
\sqrt{\frac{\pi-\ep_1}{2\pi}} 
V(D) \left(\inf_{D}|\varphi_{1\zb}|\right).
\end{align}

Now we turn to $\varphi_2$. Since $\varphi_2(z,0)=0$, for every $\ep>0$ there 
exists $\delta>0$ and $j_{\ep}$ so that
\begin{itemize}
 \item[i.]  $|\varphi_2(z,w)|<\ep$ for $(z,w)\in \Dc$ and $|w|\leq \delta$,
 \item[ii.] $|f_j(z,w)|<\ep$ for  $(z,w)\in \Dc, |w|\geq \delta$ and 
 $j\geq j_{\ep}$. 
\end{itemize}
Let us denote $\D_{1,\delta}=\{(z,w)\in \D:|w|< \delta\}$ and 
$\D_{2,\delta}=\{(z,w)\in \D:|w|> \delta\}$. Then for $j\geq  j_{\ep}$ we have 
\begin{align*}
 \left\|H_{\varphi_2}f_j\right\| &\leq \|\varphi_2f_j\|\\
 &=\|\varphi_2f_j\|_{\D_{1,\delta}}+ \|\varphi_2f_j\|_{\D_{2,\delta}}\\
 &\leq \ep(\|f_j\|+\|\varphi_2\|)\\
 &\leq \ep (\pi \tau_{\D}^{2-\alpha_j}+\|\varphi_2\|).
\end{align*}
Then, $\limsup_{j\to \infty}\left\|H_{\varphi_2}f_j\right\|\leq \ep(\pi 
\tau_{\D}+\|\varphi_2\|)$. Since $\ep$ is arbitrary, we get
\begin{align}\label{Eqn3}
\lim_{j\to \infty}\left\|H_{\varphi_2}f_j\right\|= 0. 
\end{align}
By the definition of essential norms for Hankel operators, for any $\ep_2>0$ 
there exists a compact operator $K_{\ep_2}:A^2(\D)\to L^2(\D)$ such that 
\[\left\|H_{\varphi}\right\|_e
\geq \left\|H_{\varphi}-K_{\ep_2} \right\|-\ep_2.\] 
Then 
\begin{align}
\nonumber \left\|H_{\varphi}\right\|_e &\geq 
\limsup_{j\to\infty}\frac{\left\|H_{\varphi}f_j-K_{\ep_2} 
f_j\right\|}{\pi \tau_{\D}^{2-\alpha_j}}-\ep_2\\
 \label{Eqn4}&\geq \limsup_{j\to\infty}\frac{\left\|H_{\varphi_1}f_j\right\|- 
\left\|H_{\varphi_2}f_j\right\| -\left\|K_{\ep_2} f_j\right\|}{\pi 
\tau_{\D}^{2-\alpha_j}}-\ep_2 
\\
\nonumber &=\limsup_{j\to\infty}\frac{\left\|H_{ \varphi_1}f_j\right\|}{\pi 
\tau_{\D}^{2-\alpha_j}}-\ep_2.
\end{align}
In the last equality we used \eqref{Eqn3}, compactness of 
$K_{\ep_2}$, and the fact that $f_j\to 0$ weakly.  
 Therefore, combining \eqref{Eqn2} and  \eqref{Eqn4} together with the 
fact that the constants $\ep_1,\ep_2>0$ are arbitrary we get 
\[ \left\|H_{\varphi}\right\|_e \geq 
\frac{1}{\sqrt{2}\pi \tau_{\D}}\sup_{D\subset b\D}\{V(D)
\inf_{\xi\in D}\{ |\varphi_{\zb}(\xi)| \} \}. \] 
We note that there is a one-to-one correspondence between the (affine) disks 
in $b\D$ and $F\in \Gamma_{b\D}$. Since we need $F:\mathbb{D}\to D$ to be a 
surjection  we must have $F(\xi)=(\tau_0\xi,0)$. Then one can show that 
\[V(D)\inf_{\xi\in D}\{ |\varphi_{\zb}(\xi)| \}= \pi|F'(0)| \inf_{\xi\in
\mathbb{D}}\{ |(\varphi \circ F)_{\xb}(\xi)| \}.\]
Therefore, we have 
\[ \left\|H_{\varphi}\right\|_e \geq 
\sup_{F\in\Gamma_{b\D}} \left\{\frac{|F'(0)|}{\sqrt{2}\tau_{\D}} 
\inf_{\xi\in\mathbb{D}}\left\{ |(\varphi \circ F)_{\xb}(\xi)| \right\} 
\right\}.\] 

Now we turn to the upper estimate.  Let $\rho$ be a defining function 
for $\D$. That is, $\rho$ is a $C^1$-smooth function in a neighborhood of 
$\Dc$  such that $\rho<0$ on $\D$, $\rho>0$ on $\C^2\setminus \Dc$, and 
$|\nabla \rho|\neq 0$ on $b\D$.  Then we define the complex tangential and 
complex normal vector fields as 
\begin{align*}
 L_1=&\frac{2\sqrt{2}}{\|\nabla \rho\|}\left(\frac{\partial \rho}{\partial 
w}\frac{\partial}{\partial 
z}-\frac{\partial \rho}{\partial z}\frac{\partial}{\partial w}\right)\\ 
 L_2=&\frac{2\sqrt{2}}{\|\nabla \rho\|}\left(\frac{\partial \rho}{\partial 
\zb}\frac{\partial}{\partial z}+
\frac{\partial \rho}{\partial \wb}\frac{\partial}{\partial w}\right).
 \end{align*}
 One can check that $\{L_1,L_2\}$ form a continuous orthonormal basis for 
the space of $(1,0)$ type vector fields on a neighborhood on $b\D$.
Let $\omega_1$ and $\omega_2$ be the differential forms of type $(1,0)$ that 
are the dual to $L_1$ and $L_2$, respectively.  That is, 
\begin{align*}
\omega_1=& \frac{\sqrt{2}}{\|\nabla \rho\|}\left(\frac{\partial \rho}{\partial 
\wb}dz-\frac{\partial \rho}{\partial \zb}dw\right)\\
\omega_2=&\frac{\sqrt{2}}{\|\nabla \rho\|}\left(\frac{\partial \rho}{\partial 
z}dz+\frac{\partial \rho}{\partial w}dw\right).
\end{align*}
One can check that $\|\omega_1\|=\|\omega_2\|=1$ and  $\dbar 
f=\overline{L}_1(f)\overline{\omega}_1+\overline{L}_2(f)\overline{\omega}_2$ 
for any $f\in C^1(\Dc)$ (see special boundary charts in 
\cite[p. 12]{StraubeBook}).

Using the method in the first part of the proof of Theorem 3 in 
\cite[p. 3739--3740]{CuckovicSahutoglu09} 
($\widetilde{\beta}$ and $\widehat{\beta}$ in \cite{CuckovicSahutoglu09} 
correspond to $\varphi_3$ and $\varphi_4$ below, respectively) we define 
$\varphi_3,\varphi_4\in C^1(\Dc)$ such that 
\begin{itemize}
 \item[i.] $\varphi=\varphi_3+\varphi_4$,
 \item[ii.] $\varphi_3=\varphi$  and $\overline{L}_2(\varphi_3)=0$ on $b\D$,
 \item[iii.] $\varphi_4=0$ on $b\D$. 
\end{itemize}
We note that $\varphi_4$ is a uniform limit of compactly supported smooth 
functions on $\D$. This fact together with Montel's Theorem  
imply that $H_{\varphi_4}$ is a limit of compact operators in the operator 
norm. Hence $H_{\varphi_4}$ is compact and 
$\|H_{\varphi}\|_e=\|H_{\varphi_3}\|_e$.  

Let 
\[\Pi=\overline{\bigcup_{F\in \Gamma_{b\D}}F(\mathbb{D})}\]
and $\chi_{\ep}\in C^{\infty}(\Dc)$ 
such that $0\leq \chi_{\ep}\leq 1, \chi_{\ep}=1$ on $\Pi_{\ep}=\{z\in \Dc: 
d(z,\Pi)\leq \ep\},$ and $\chi_{\ep}=0$ on $\Dc\setminus \Pi_{2\ep}$. 
Then for $f\in A^2(\D)$ we have 
\[H_{\varphi_3}= \dbar^*N M_{\dbar\varphi_3}= \dbar^*N 
M_{\chi_{\ep}\dbar\varphi_3}+ \dbar^*N M_{(1-\chi_{\ep})\dbar\varphi_3}\]
where $M_h$ denotes the multiplication by $h$. 
First we will show that $\dbar^*N M_{(1-\chi_{\ep})\dbar\varphi_3}$ is compact 
on $A^2(\D)$. Let $f\in A^2(\D)$. 
\begin{align*}
 \|\dbar^*N f(1-\chi_{\ep})\dbar\varphi_3\|^2&=\langle \dbar^*N 
f(1-\chi_{\ep})\dbar\varphi_3, \dbar^*N f(1-\chi_{\ep})\dbar\varphi_3\rangle\\
&=\langle f \dbar\varphi_3 , (1-\chi_{\ep})N\dbar \dbar^*N 
f(1-\chi_{\ep})\dbar\varphi_3\rangle \\
&\lesssim \|f\| \|(1-\chi_{\ep})N\dbar \dbar^*N 
f(1-\chi_{\ep})\dbar\varphi_3\|
\end{align*}
Now we will use the fact that  $(1-\chi_{\ep})N$ is compact. This is 
essentially done on pages 3740--3741 in the proof of Theorem 3 in 
\cite{CuckovicSahutoglu09}. The idea is to use compactness of  $\dbar$-Neumann 
operator  locally  to get the following compactness estimate: for every 
$\ep_1>0$ there exists a compact operator $K_{\ep_1}$ on $L^2_{(0,1)}(\D)$ so 
that 
\[\|(1-\chi_{\ep})Nh\|\leq \ep\|h\|+\|K_{\ep_1}h.\|\]
Then using the fact that  $\dbar\dbar^* N$ is a bounded operator in the second 
inequality below,  we get 
\begin{align*}
\|(1-\chi_{\ep})N\dbar \dbar^*N f(1-\chi_{\ep})\dbar\varphi_3\| 
\leq & \ep_1\| \dbar\dbar^* N f(1-\chi_{\ep})\dbar\varphi_3\| 
+\|K_{\ep_1}\dbar \dbar^* Nf(1-\chi_{\ep})\dbar\varphi_3\|  \\ 
\lesssim & \ep_1\|f\|+\|\widetilde{K}_{\ep_1} f\| _{(0,1)}
\end{align*}
where $\widetilde{K}_{\ep_1}=K_{\ep_1}\dbar\dbar^* 
NM_{(1-\chi_{\ep})\dbar\varphi_3}$ is a compact operator. Therefore, 
$\dbar^*N M_{(1-\chi_{\ep})\dbar\varphi_3}$ satisfies a compactness estimate 
and hence it is compact. 
Then 
\[\dbar \varphi_3=\overline{L}_1(\varphi_3)\overline{\omega}_1+\overline{L}
_2(\varphi_3)\overline { \omega} _2.\]
Using the facts that $\overline{L}_2\varphi_3=0$ and $\varphi=\varphi_3$ on 
$b\Dc$ we get 
\[|\dbar\varphi_3|=|\overline{L}_1(\varphi_3)|=|\overline{L}_1(\varphi)| 
\text{ on } b\D.\]  
Therefore, we have 
\[\| \dbar^*N f\chi_{\ep}\dbar\varphi_3\|\leq \|\dbar^*N\| 
\|f\chi_{\ep}\dbar\varphi_3\|\leq \|\dbar^*N\|  
\sup\left\{|\overline{L}_1(\varphi)(z)|:z\in \Pi_{2\ep}\right\} \|f\|\]
So if we let $\ep$ go to zero and use the fact that 
$\|\dbar^*N\|\leq \sqrt{e}\tau_{\D}$, we get 
\[\left\|H_{\varphi}\right\|_e\leq \sqrt{e}\tau_{\D} 
\sup\left\{|\overline{L}_1(\varphi)(z)|:z\in \Pi\right\}.\]
On the other hand, for $p\in \Pi$ there exist $p_j\in \Pi, \xi_j\in \mathbb{D}$, 
and $F_j\in \Gamma_{b\D}$ such that $F_j(\xi_j)=p_j$ and $\lim p_j= p$. We note 
that if $p$ is not on the boundary of a disk then we can choose $p_j=p$ for all $j$. 

Let $F_j(\xi)=(F_{j1}(\xi),F_{j2}(\xi))$ for $\xi\in \mathbb{D}$. Since $\D$ is 
convex in $\C^2$ and we assume that $p_j$ is in a horizontal disk, $F_{j1}$ is 
linear and $F_{j2}$ is constant. The chain rule and the fact that $L_1$ is the 
complex tangential derivative imply that
\[(\varphi \circ F_j)_{\xb}(\xi_j)
=\varphi_{\zb}(p_j)\overline{F_{j1\xi}(\xi_j)}
=\overline{L}_1(\varphi)(p_j)\overline{F'_{j1}(\xi_j)}
=\overline{L}_1(\varphi)(p_j)\overline{F'_{j1}(0)}.\]
Hence  
\[|\overline{L}_1(\varphi)(p_j)|
=\frac{\left|(\varphi\circ F_j)_{\xb}(\xi_j)\right|}{\left|\overline{F'_{j}(0)}
\right|}.\] 
Then, if we take supremum over $j$ we get 
\[|\overline{L}_1(\varphi)(p)| \leq 
\sup_{j}\sup_{\xi \in \mathbb{D}}\left\{ \frac{\left|(\varphi \circ 
F_j)_{\xb}(\xi)\right|}{\left| F'_j(0)\right|} \right\}.\]
Hence, we have 
\[\left\|H_{\varphi}\right\|_e\leq \sup_{F\in \Gamma_{b\D}}\left\{ 
\frac{\sqrt{e}\tau_{\D}}{|F'(0)|}  
\sup_{\xi\in \mathbb{D}}\left\{ |(\varphi \circ F)_{\xb}(\xi)| 
\right\} \right\}.\]
This completes the proof of Theorem \ref{ThmMain}.
\end{proof}
\begin{proof}[Proof of Theorem \ref{ThmBidisk}]
The proof of Theorem \ref{ThmBidisk} is very similar to the first part of the 
proof of Theorem \ref{ThmMain}. So instead of running through the whole 
argument again we will point out where they differ and the modifications 
needed for this proof. Without of loss of generality we may assume that there 
exists $z_0\in \mathbb{D},p\in b\mathbb{D}$ such that $\varphi_{\zb}(z_0,p)\neq 
0$. In this case wedge $W$ is replaced by the disk $\mathbb{D}$ in $w$. Let us 
choose 
a sequence $\{p_j\}\subset \mathbb{D}$ such that $\lim_{j\to \infty}p_j=p$. 
Let us define $f_j(w)=k_{p_j}(w)$ where $k_{p_j}$ is the normalized Bergman 
kernel of $\mathbb{D}$ centered at $p_j$. Then instead of 
\eqref{Eqn0} we have 
\[\|f_j\|_{\mathbb{D}}=1 \text{ and } \|f_j\|=\sqrt{\pi}\]
The decomposition of $\varphi$ is unnecessary in case of the bidisk. Or simply 
we decompose $\varphi=\varphi_1+\varphi_2$ where  $\varphi_1=\varphi$ and 
$\varphi_2=0$.  We choose $D\subset \mathbb{D}\times \{p\}$ such that  
$(z_0,p)\in D$ and $\varphi_{\zb}$ does not vanish on $D$. In a similar fashion 
as in the proof of Theorem 
\ref{ThmMain} we get the following inequality. 
\[\|f_j\|_{\mathbb{D}} \frac{\int_{D}\chi(z) dV(z)}{\|\gamma_{\zb}\|_{D}} 
\leq \left\|H_{\varphi}f_j\right\|_{D\times \mathbb{D}}
\leq \left\|H_{\varphi}f_j\right\|.\]
Then 
\[\left\|H_{\varphi}f_j\right\| \geq 
\frac{V(D)}{\sqrt{2\pi}} \left(\inf_{D}|\varphi_{\zb}|\right).\]
We can estimate the essential norm as in \eqref{Eqn4}
\[\|H_{\varphi}\|_e\geq \limsup_{j\to 
\infty}\frac{\|H_{\varphi}f_j\|}{\sqrt{\pi}}-\ep\]
for an arbitrary $\ep>0$. Furthermore,  we choose $r>0$ so that  
$F(\xi)=(r(\xi-z_0),p)$ and $D=F(\mathbb{D})$. Then   
\[V(D)\inf_{\xi\in D}\{ |\varphi_{\zb}(\xi)| \}= \pi|F'(0)| \inf_{\xi\in
\mathbb{D}}\{ |(\varphi \circ F)_{\xb}(\xi)| \}.\]
Hence 
\[\|H_{\varphi}\|_e\geq \limsup_{j\to 
\infty}\frac{\|H_{\varphi_1}f_j\|}{\sqrt{\pi}}-\ep \geq 
\frac{|F'(0)|}{\sqrt{2}}
\inf_{\xi\in \mathbb{D}}\left\{ |(\varphi \circ F)_{\xb}(\xi)| \right\}-\ep.\]
Now we take supremum over $F$ and let $\ep$ go to zero 
\[ \left\|H_{\varphi}\right\|_e \geq \sup_{F\in 
\Gamma_{b\mathbb{D}^2}}\left\{\frac{|F'(0)|}{\sqrt{2}}
\inf_{\xi\in \mathbb{D}}\left\{ |(\varphi \circ F)_{\xb}(\xi)| \right\} 
\right\}.\]
This completes the proof of Theorem \ref{ThmBidisk}.
\end{proof}

\end{document}